\documentclass[leqno,11pt]{article}
\usepackage{amssymb,amsmath,mathabx}

\usepackage{amsfonts,euscript,epsfig,color}
\usepackage[all]{xy}

\newtheorem{theo}{Theorem}[section]
\newtheorem{defi}[theo]{Definition}

\newtheorem{lem}[theo]{Lemma}
\newtheorem{prop}[theo]{Proposition}
\newtheorem{rem}[theo]{Remark}
\newtheorem{coro}[theo]{Corollary}
\newtheorem{exam}[theo]{Example}
\newtheorem{conj}[theo]{Conjecture}

\newenvironment{proof}{{\bf Proof.}}

\newcommand{\kgot}{\ensuremath{\mathfrak{k}}}
\newcommand{\hgot}{\ensuremath{\mathfrak{h}}}
\newcommand{\ggot}{\ensuremath{\mathfrak{g}}}

\newcommand{\tgot}{\ensuremath{\mathfrak{t}}}
\newcommand{\pgot}{\ensuremath{\mathfrak{p}}}

\newcommand{\Rgot}{\ensuremath{\mathfrak{R}}}

\newcommand{\Bcal}{\ensuremath{\mathcal{B}}}
\newcommand{\Ccal}{\ensuremath{\mathcal{C}}}
\newcommand{\Dcal}{\ensuremath{\mathcal{D}}}
\newcommand{\Ecal}{\ensuremath{\mathcal{E}}}

\newcommand{\Lcal}{\ensuremath{\mathcal{L}}}

\newcommand{\Ocal}{\ensuremath{\mathcal{O}}}
\newcommand{\Pcal}{\ensuremath{\mathcal{P}}}

\newcommand{\Scal}{\ensuremath{\mathcal{S}}}
\newcommand{\Ucal}{\ensuremath{\mathcal{U}}}

\newcommand{\Cbb}{\ensuremath{\mathbb{C}}}
\newcommand{\Rbb}{\ensuremath{\mathbb{R}}}
\newcommand{\Zbb}{\ensuremath{\mathbb{Z}}}

\newcommand{\As}{{\ensuremath{\rm As}}}
\newcommand{\AS}{{\ensuremath{\rm AS}}}

\newcommand{\croc}{\ensuremath{\hookrightarrow}}
\newcommand{\T}{\ensuremath{\hbox{\bf T}}}

\newcommand{\End}{\ensuremath{\hbox{\rm End}}}

\newcommand{\indice}{\mathrm{Index}}

\def \what {\widehat}

\def \wR {{\widehat{R}}}
\def \wK {{\widehat{K}}}
\def \tK {{\tilde{K}}}

\def \wG {{\widehat{G}}}

\newcommand{\p}{\operatorname{p}}

\def \clif {\mathbf{c}}

\def \p {{\rm p}}

\def \cl {{\rm cl}}

\def \p {{\rm p}}

\newcommand{\QS}{\ensuremath{\mathcal{Q}^{\mathrm{spin}}}}
\newcommand{\QSf}{\ensuremath{\widehat{\mathrm{Q}}^{\mathrm{spin}}}}

\newcommand{\Spin}{\ensuremath{{\rm Spin}}}
\newcommand{\spinc}{\ensuremath{{\rm spin}^{c}}}
\newcommand{\Spinc}{\ensuremath{{\rm Spin}^{c}}}

\begin{document}

\title{Kirillov's orbit method: the case of discrete series representations}

\author{Paul-Emile PARADAN
\footnote{Institut Montpelli\'erain Alexander Grothendieck,
Universit\'e de Montpellier, CNRS \texttt{paul-emile.paradan@umontpellier.fr}}}


\maketitle

\begin{abstract}
Let $\pi$ be a discrete series representation  of a real semi-simple Lie group $G'$ and let $G$ be a semi-simple subgroup of $G'$. In this paper, we give a geometric expression of the $G$-multiplicities in $\pi\vert_G$ when the representation $\pi$ 
is $G$-admissible.
\end{abstract}

{\small
\tableofcontents}

\section{Introduction}

This paper is concerned by a central problem of non-commutative harmonic analysis : given a unitary irreducible 
representation $\pi$ of a Lie group $G'$, how does $\pi$ decomposes when restricted to a closed subgroup 
$G\subset G'$ ? We analyse this problem for Harish-Chandra discrete series representations of a connected 
real semi-simple Lie group $G'$ with finite center, relatively to a connected real semi-simple subgroup $G$ (also with finite center).

We start with Harish-Chandra parametrization of the discrete series representations. We can attach an unitary irreducible 
representation $\pi^{G'}_{\Ocal'}$ of the group $G'$ to any regular admissible elliptic coadjoint orbit $\Ocal'\subset (\ggot')^*$, 
and Schmid proved that the representation $\pi^{G'}_{\Ocal'}$ could be realize as the quantization of the orbit 
$\Ocal'$ \cite{Schmid71,Schmid76}. This is a vast generalization of Borel-Weil-Bott's construction of finite dimensional representations of compact Lie groups. 
In the following, we denote $\wG_d$ and $\wG'_d$ the sets of regular admissible elliptic coadjoint orbits of our connected 
real semi-simple Lie groups $G$ and $G'$.

One of the rule of Kirillov's orbit method \cite{Kirillov62} is concerned with the functoriality relatively to inclusion $G\croc G'$ of 
closed subgroups. It means that, starting with discrete series representations representations $\pi^G_{\Ocal}$ and 
$\pi^{G'}_{\Ocal'}$ attached  to regular admissible elliptic orbits $\Ocal\subset\ggot^*$ and 
$\Ocal'\subset (\ggot')^*$, one expects that the multiplicity of $\pi^{G}_{\Ocal}$ in the restriction $\pi^{G'}_{\Ocal'}\vert_{G}$ can be computed geometrically in terms of the space 
\begin{equation}\label{eq:reduced-space-intro}
\Ocal'\slash\!\!\slash \Ocal := \Ocal'\cap\p_{\ggot,\ggot'}^{-1}(\Ocal)/G,
\end{equation} 
where $\p_{\ggot,\ggot'}:(\ggot')^*\to\ggot^*$ denotes the canonical projection. One recognises  
that  (\ref{eq:reduced-space-intro}) is a symplectic reduced space in the sense of Marsden-Weinstein, since  
$\p_{\ggot,\ggot'}:\Ocal'\to \ggot^*$ is the moment map relative to the Hamiltonian action of $G$ on $\Ocal'$. 

In other words, Kirillov's orbit method tells us that the branching laws $[\pi^{G}_{\Ocal}:\pi^{G'}_{\Ocal'}]$ should be 
compute geometrically. So far, the following special cases have been achieved~:

1. $G\subset G'$ are compact. In the 1980s, Guillemin and Sternberg \cite{Guillemin-Sternberg82} studied the geometric quantization of general 
$G$-equivariant compact K\"{a}hler manifolds. They  
proved the ground-breaking result that the multiplicities of this $G$-representation are calculated in terms of geometric 
quantizations of the symplectic reduced spaces. This phenomenon, which has been the center of many research and generalisations 
\cite{Meinrenken98,Meinrenken-Sjamaar,Tian-Zhang98,pep-RR,Ma-Zhang14, pep-IF,pep-vergne-acta,pep-vergne:witten,Hochs-Song-duke}, is 
called nowaday ``quantization commutes with reduction'' (in short, ``[Q,R]=0''). 

2. $G$ is a compact subgroup of $G'$. In \cite{pep-ENS}, we used the Blattner formula to see that the [Q,R]=0 phenomenon holds in this context 
when $G$ is a maximal compact subgroup. Duflo-Vergne have generalized this result for any compact subgroup \cite{Duflo-Vergne2011}. Recently, 
Hochs-Song-Wu have shown that the [Q,R]=0 phenomenon holds for any tempered representation of $G'$ relatively to a maximal compact subgroup 
\cite{Hochs-Song-Wu}.

3. $\pi^{G'}_{\Ocal'}$ is an holomorphic discrete series. We prove that the [Q,R]=0 phenomenon holds with some assumption on $G$ \cite{pep-jems}.

\medskip

However, one can observe that the restriction of $\pi^{G'}_{\Ocal'}$ with respect to $G$ may have a wild behavior in general, even if $G$ is a maximal reductive subgroup in $G'$ (see  \cite{Toshi-Inventiones94}). 

In \cite{Toshi-Inventiones94,Toshi-Annals98,Toshi-Inventiones98} T. Kobayashi singles out a nice class of branching problems 
where each $G$-irreducible summand of $\pi\vert_G$ occurs 
discretely with finite multiplicity~: the restriction $\pi\vert_G$ is called $G$-admissible.

So we focus our attention to a discrete series $\pi^{G'}_{\Ocal'}$ that admit an admissible restriction relatively to $G$. 
It is well-known that we have then an Hilbertian direct sum decomposition
$$
\pi_{\Ocal'}^{G'}\vert_G=\sum_{\Ocal\in \widehat{G}_d} m_{\Ocal'}^{\Ocal}\,\pi_\Ocal^G
$$
where the multiplicities $m_{\Ocal'}^{\Ocal}$ are finite.

We will use the following geometrical characterization of the 
$G$-admissibility obtained by Duflo and Vargas \cite{Duflo-Vargas2007,Duflo-Vargas2010}.

\begin{prop}\label{prop:Duflo-Vargas}
The representation $\pi_{\Ocal'}^{G'}$ is $G$-admissible if and only if the restriction of the map 
$\p_{\ggot,\ggot'}$ to the coadjoint orbit $\Ocal'$  is a proper map.
\end{prop}


Let $(\Ocal',\Ocal)\in \widehat{G}'_d\times \wG_d$. Let us explain how we can quantize the compact symplectic reduced space 
$\Ocal'\slash\!\!\slash \Ocal$ when the map $\p_{\ggot,\ggot'}:\Ocal'\to \ggot^*$ is proper. 

If $\Ocal$ belongs to the set of regular values of $\p_{\ggot,\ggot'}:\Ocal'\to \ggot^*$, then $\Ocal'\slash\!\!\slash \Ocal$ 
is a compact symplectic orbifold equipped with a $\spinc$ structure. We denote $\QS(\Ocal'\slash\!\!\slash \Ocal)\in \Zbb$ the index of the corresponding $\spinc$ Dirac operator.

In general, we consider an elliptic coadjoint $\Ocal_\epsilon$ closed enough\footnote{The precise meaning will be explain in Section 
\ref{sec:quantization-proper-2}.} 
to $\Ocal$, so that $\Ocal'\slash\!\!\slash \Ocal_\epsilon$ is a compact 
symplectic orbifold equipped with a $\spinc$ structure. Let $\QS(\Ocal'\slash\!\!\slash \Ocal_\epsilon)\in \Zbb$ be the index of the corresponding 
$\spinc$ Dirac operator. The crucial fact is that the quantity $\QS(\Ocal'\slash\!\!\slash \Ocal_\epsilon)$ does not depends on the 
choice of generic and small enough $\epsilon$. Then we take
$$
\QS(\Ocal'\slash\!\!\slash \Ocal):=\QS(\Ocal'\slash\!\!\slash \Ocal_\epsilon)
$$
 for generic and small enough $\epsilon$.

\medskip

The main result of this article is the following

\begin{theo}\label{theo:main}
 Let $\pi_{\Ocal'}^{G'}$ be a discrete series representation of $G'$ attached to a regular admissible elliptic coadjoint orbits $\Ocal'$. 
 If $\pi_{\Ocal'}^{G'}$ is $G$-admissible we have the Hilbertian 
 direct sum
\begin{equation}\label{eq:QR-G}
\pi_{\Ocal'}^{G'}\vert_G=\sum_{\Ocal\in \widehat{G}_d} \QS(\Ocal'\slash\!\!\slash \Ocal)\, \pi^G_\Ocal.
\end{equation}
In other words the multiplicity $[\pi^G_\Ocal:\pi_{\Ocal'}^{G'}]$ is equal to $\QS(\Ocal'\slash\!\!\slash \Ocal)$.
\end{theo}

In a forthcoming paper we will study Equality (\ref{eq:QR-G}) in further details when $G$ is a symmetric subgroup of $G'$.

Theorem \ref{theo:main} give a positive answer to a conjecture of Duflo-Vargas.

\begin{theo}\label{theo:conjecture-Duflo-Vargas}
Let $\pi_{\Ocal'}^{G'}$ be a discrete series representation of $G'$ that is $G$-admissible. Then all the representations $\pi_\Ocal^G$ 
which occurs in $\pi_{\Ocal'}^{G'}$ belongs to a {\em unique} family of discrete series representations of $G$.
\end{theo}


\section{Restriction of discrete series representations}

Let $G$ be a connected real semi-simple Lie group $G$ with finite center. A discrete series representation of $G$ is an 
irreducible unitary representation that is isomorphic to a sub-representation of the left regular representation in 
$L^2(G)$. We denote $\wG_d$ the set of isomorphism class of discrete series representation of $G$. 

We know after Harish-Chandra that $\wG_d$ is non-empty only if $G$ has a compact Cartan subgroup. We denote 
$K\subset G$ a maximal compact subgroup and we suppose that $G$ admits a compact Cartan subgroup $T\subset K$.  
The Lie algebras of the groups $T,K,G$ are denoted respectively $\tgot$, $\kgot$ and $\ggot$.

In this section we recall well-know facts concerning restriction of discrete series representations.

\subsection{Admissible coadjoint orbits}

Here we recall the parametrization of $\wG_d$ in terms of regular admissible elliptic coadjoint orbits. Let us fix some notations. 
We denote $\Lambda\subset \tgot^*$ the weight lattice: any $\mu\in\Lambda$ defines a $1$-dimensional representation 
$\Cbb_\mu$ of the torus $T$.

Let $\Rgot_c\subset\Rgot\subset\Lambda$ be respectively the set of 
(real) roots for the action of $T$ on $\kgot\otimes\Cbb$ and 
$\ggot\otimes\Cbb$. The non-compact roots are those belonging to the set $\Rgot_n:=\Rgot\setminus \Rgot_c$. 
We choose a system of positive roots $\Rgot_{c}^{+}$ for 
$\Rgot_{c}$, we denote by $\tgot^{*}_{+}$ the corresponding Weyl chamber. Recall that
$\Lambda\cap \tgot^*_+$ is the set of dominant weights.

We denote by $B$ the Killing form on $\ggot$. It induces a scalar product 
(denoted by $(-,-)$) on $\tgot$, and then on $\tgot^{*}$. An element $\lambda\in \tgot^{*}$ is called $G$-{\em regular} if 
$(\lambda,\alpha)\neq 0$ for every $\alpha\in\Rgot$, or equivalently,  
if the stabilizer subgroup of $\lambda$ in $G$ is $T$.  For any $\lambda\in\tgot^*$ we denote 
$$
\rho(\lambda):=\frac{1}{2}\sum_{\alpha\in \Rgot,(\alpha,\lambda)>0}\alpha.
$$
We denote also $\rho_c:=\frac{1}{2}\sum_{\alpha\in \Rgot^+_c}\alpha$.

\begin{defi}
\begin{enumerate}
\item A coadjoint orbit $\Ocal\subset\ggot^*$ is elliptic if $\Ocal\cap\tgot^*\neq\emptyset$.

\item An elliptic coadjoint orbit $\Ocal$ is {\rm admissible}\footnote{Duflo has defined a notion of 
admissible coadjoint orbits in a much broader context \cite{Duflo80}.} when $\lambda-\rho(\lambda)\in \Lambda$ for any 
$\lambda\in \Ocal\cap\tgot^*$.
\end{enumerate}
\end{defi}

Harish-Chandra has parametrized $\wG_d$ by the set of regular admissible elliptic coadjoint orbits of $G$. In order to simplify 
our notation, we denote $\wG_d$ the set of regular admissible elliptic coadjoint orbits. For an orbit 
$\Ocal\in \widehat{G}_d$  we denote $\pi_\Ocal^G$ the corresponding discrete series representation of $G$.

Consider the subset $(\tgot^*_+)_{se}:=\{\xi\in \tgot^*_+, (\xi,\alpha)\neq 0,\ \forall \alpha\in\Rgot_n\}$ of the Weyl chamber. 
The subscript means {\em strongly elliptic}, see Section \ref{sec:proper-2}. By definition any $\Ocal\in \wG_d$ intersects $(\tgot^*_+)_{se}$ 
in a unique point.

\begin{defi}\label{defi:chamber}The connected component $(\tgot^*_+)_{se}$ are called chambers. If 
$\Ccal$ is a chamber, we denote $\wG_d(\Ccal)\subset \wG_d$ the subset of regular admissible elliptic orbits 
intersecting $\Ccal$.
\end{defi}

Notice that the Harish-Chandra parametrization has still  a meaning when $G=K$ is a compact connected Lie group. In this case 
$\wK$ corresponds to the set of regular admissible coadjoint orbits $\Pcal\subset \kgot^*$, i.e. those of the form 
$\Pcal=K\mu$ where $\mu- \rho_c\in\Lambda\cap\tgot^*_+$:  the corresponding representation $\pi^K_{\Pcal}$ is the 
irreducible representation of $K$ with highest weight $\mu-\rho_c$.

\subsection{Spinor representation}\label{sec:spinor-representation}
Let $\pgot$ be the orthogonal complement of $\kgot$ in $\ggot$: the Killing form of $\ggot$ defines a $K$-invariant 
Euclidean structure on it. Note that $\pgot$ is even dimensional since the groups $G$ and $K$ have the same rank. 

We consider the two-fold cover ${\rm Spin}(\pgot)\to{\rm SO}(\pgot)$ and the morphism $K\to {\rm SO}(\pgot)$. 
We recall the following basic fact.

\begin{lem}\label{lem:tilde-K}
There exists a unique covering $\tK\to K$ such that
\begin{enumerate}
\item $\tK$ is a compact connected Lie group,
\item the morphism $K\to {\rm SO}(\pgot)$ lifts to a morphism $\tilde{K}\to  {\rm Spin}(\pgot)$.
\end{enumerate}
\end{lem}

Let $\xi\in\tgot^*$ be a regular element and consider 
\begin{equation}\label{eq:rho-n-xi}
\rho_n(\xi):=\frac{1}{2}\sum_{\alpha\in \Rgot_n,(\alpha,\xi)>0}\alpha.
\end{equation}
Note that 
\begin{equation}\label{eq:tilde-lambda}
\tilde{\Lambda}=\Lambda \,\bigcup\, \{\rho_n(\xi)+\Lambda\}
\end{equation}
is  a lattice that does not depends on the choice of $\xi$.

Let $T\subset K$ be a maximal torus and $\tilde{T}\subset\tK$ be the pull-back of $T$ relatively to the covering 
 $\tilde{K}\to K$. We can now precise Lemma \ref{lem:tilde-K}.
\begin{lem}\label{lem:tilde-K-precise}
Two situations occur: 
\begin{enumerate}
\item if $\rho_n(\xi)\in\Lambda$ then $\tK\to K$ and $\tilde{T}\to T$ are isomorphisms, and $\tilde{\Lambda}=\Lambda$.
\item if $\rho_n(\xi)\notin\Lambda$ then $\tK\to K$ and $\tilde{T}\to T$ are two-fold covers, and $\tilde{\Lambda}$ is the lattice of weights for $\tilde{T}$.
\end{enumerate}
\end{lem}

\medskip

Let $\Scal_\pgot$ the spinor representation of the group $\Spin(\pgot)$. Let $\clif: {\rm Cl}(\pgot)\to\End_\Cbb(\Scal_\pgot)$ be the Clifford action. Let $o$ be an orientation on $\pgot$. If $e_1,e_2,\cdots,e_{\dim\pgot}$ is an oriented 
orthonormal base of $\pgot$ we define the element 
$$
\epsilon_o:= (i)^{\dim \pgot/2}e_1e_2 \cdots e_{\dim\pgot}\in {\rm Cl}(\pgot)\otimes \Cbb.
$$
that depends only of the orientation. We have $\epsilon_o^2=-1$ and $\epsilon_o v=-v\epsilon$ for any $v\in\pgot$. 
The element $\clif(\epsilon_o)$ determines 
 a decomposition $\Scal_\pgot=\Scal_\pgot^{+,o}\oplus \Scal_\pgot^{-,o}$ into irreducible representations $\Scal_\pgot^{\pm,o}= \ker (\clif(\epsilon_o)\mp Id)$ of $\Spin(\pgot)$. We denote
$$
\Scal_\pgot^{o}:= \Scal_\pgot^{+,o}\ominus \Scal_\pgot^{-,o}
$$
the corresponding virtual representation of $\tilde{K}$. 

\begin{rem}If $o$ and $o'$ are two orientations on $\pgot$, we have 
$\Scal_\pgot^{o}=\pm\Scal_\pgot^{o'}$, where the sign $\pm$ is the ratio between $o$ and $o'$.
\end{rem}

\begin{exam}\label{exam:orientation}
Let $\lambda\in \kgot$ such that the map ${\rm ad}(\lambda):\pgot\to\pgot$ 
is one to one. We get a symplectic form $\Omega_\lambda$ on $\pgot$ defined by the relations 
$\Omega_\lambda(X,Y)=\langle\lambda,[X,Y]\rangle$ for $X,Y\in\pgot$. We denote $o(\lambda)$ be the orientation of 
$\pgot$ defined by the top form $\Omega_\lambda^{\dim \pgot/2}$.
\end{exam}

%
%
%

\subsection{Restriction to the maximal compact subgroup}

We start with a definition.

\begin{defi}
$\bullet$ We denote $\what{R}(G,d)$ the group formed by the formal (possibly infinite) sums
$$
\sum_{\Ocal\in \widehat{G}_d} \,a_\Ocal\, \pi_\Ocal^G
$$
where $a_\Ocal\in\Zbb$.

$\bullet$ Similarly we denote $\what{R}(K)$ the group formed by the formal (possibly infinite) sums
$\sum_{\Pcal\in \widehat{K}} \,a_{\Pcal}\, \pi_{\Pcal}^K$
where $a_{\Pcal}\in\Zbb$.
\end{defi}

The following technical fact will be used in the proof of Theorem \ref{theo:main}.

\begin{prop}\label{prop:restriction}Let $o$ be an orientation on $\pgot$.

$\bullet$ The restriction morphism $V\mapsto V\vert_K$ defines a map $\what{R}(G,d)\to \what{R}(K)$.

$\bullet$ The map ${\bf r}^o: \what{R}(G,d)\to \what{R}(\tilde{K})$ defined by 
${\bf r}^o(V):= V\vert_K\otimes \Scal_\pgot^o$ is one to one.
\end{prop}

\begin{proof} When $\Ocal=G\lambda\in \what{G}_d$, with $\lambda\in\tgot^*$, we denote $c^G_\Ocal=\|\lambda+\rho(\lambda)\|$. Similarly 
when $\Pcal=K\mu\in \what{K}$, with $\mu-\rho_c\in\Lambda\cap\tgot^*_+$, we denote $c^K_{\Pcal}=\|\mu +\rho_c\|$. Note that for each $r>0$ the set 
$\{\Ocal\in \what{G}_d,\, c^G_\Ocal\leq r\}$ is finite.

Consider now the restriction of a discrete series representation $\pi_\Ocal^G$ relatively to $K$. The Blattner's formula 
\cite{Hecht-Schmid} tells us that the restriction $\pi_\Ocal^G\vert_K$ admits a decomposition
$$
\pi_\Ocal^G\vert_K=\sum_{\Pcal\in\what{K}}\, m_\Ocal(\Pcal) \, \pi^K_{\Pcal}
$$
where the (finite) multiplicities $m_\Ocal(\Pcal)$ are non-zero only if $c^K_{\Pcal}\geq c^G_\Ocal$. 

Consider now an element $V=\sum_{\Ocal\in \what{G}_d} a_\Ocal \,\pi^G_\Ocal \in \wR(G,d)$. The multiplicity of 
$\pi_{\Pcal}^K$ in $V\vert_K$ is equal to 
$$
\sum_{\Ocal\in \what{G}_d} a_\Ocal\, m_\Ocal(\Pcal).
$$
Here the sum admits a finite number of non zero terms since $m_\Ocal(\Pcal)=0$ if 
$c^G_\Ocal > c^K_{\Pcal}$. So we have proved that the $K$-multiplicities of 
$V\vert_K:= \sum_{\Ocal\in \what{G}_d} a_\Ocal \pi^G_\Ocal\vert_K$ 
are finite. The first point is proved.

The irreducible representation of $\tK$ are parametrized by the set $\what{\tilde{K}}$ of 
regular $\tK$-admissible coadjoint orbits $\Pcal\subset \kgot^*$, i.e. those of the form 
$\Pcal=K\mu$ where $\mu-\rho_c\in\tilde{\Lambda}\cap\tgot^*_+$. It contains the set $\what{K}$ of 
regular $K$-admissible coadjoint orbits. We define 
$$
\what{K}_{out}\subset \what{\tilde{K}}
$$
as the set of coadjoint orbits $\Pcal=K\mu$ 
where\footnote{The set $\{\rho_n(\xi)+\Lambda\}\cap\tgot^*_+$ does not depend on the choice of $\xi$.} 
$\mu-\rho_c \in \{\rho_n(\xi)+\Lambda\}\cap\tgot^*_+$.  
Here $\xi$ is any regular element of $\tgot^*$ and $\rho_n(\xi)$ is defined by (\ref{eq:rho-n-xi}). 

We notice that $\what{K}_{out}=\wK$ when $\tK\simeq K$ and that 
$\what{\tilde{K}}= \wK\cup \wK_{out}$ when $\tK\to K$ is a two-fold cover.

We will use the following basic facts.
\begin{lem}\label{lem:K-out}
${}$
\begin{enumerate}
\item $\Ocal\mapsto \Ocal_K:=\Ocal\cap\kgot^*$ defines an injective map between $\widehat{G}_d$ and 
$\wK_{out}$.
\item We have $\pi_\Ocal^G\vert_K\otimes \Scal_\pgot^o=\pm \pi^{\tilde{K}}_{\Ocal_K}$ for all $\Ocal\in\wG_d$.
\end{enumerate}
\end{lem}
\begin{proof}
Let $\Ocal:=G\lambda\in \widehat{G}_d$ where $\lambda$ is a regular element of the 
Weyl chamber $\tgot^*_+$. Then $\Ocal_K=K\lambda$ and the term 
$\lambda-\rho_c$ is equal to the sum $\lambda-\rho(\lambda)+ \rho_n(\lambda)$
where $\lambda-\rho(\lambda)\in\Lambda$ and 
$\rho_n(\lambda)\in \tilde{\Lambda}$ (see (\ref{eq:tilde-lambda})), so $\lambda-\rho_c \in \{\rho_n(\xi)+\Lambda\}$. 
The element $\lambda\in \tgot^*_+$ is regular and admissible for $\tK$: this implies that 
$\lambda-\rho_c\in\tgot^*_+$. We have proved that $\Ocal_K\in \wK_{out}$.

The second point is a classical result (a generalisation is given in Theorem \ref{theo:main-2}). Let us explain the 
sign $\pm$ in the relation. Let $\Ocal\in  \widehat{G}_d$ and $\lambda\in\Ocal\cap\kgot^*$. 
Then the sign $\pm$ is the ratio between 
the orientations  $o$ and $o(-\lambda)$ of the vector space $\pgot$ (see Example \ref{exam:orientation}).
\end{proof}

\medskip

We can now finish the proof of the second point of Proposition \ref{prop:restriction}. If 
$V=\sum_{\Ocal\in \what{G}_d} a_\Ocal\, \pi^G_\Ocal \in \what{R}(G,d)$, then 
${\bf r}^o(V)=\sum_{\Ocal\in \what{G}_d} \pm\,a_\Ocal \, \pi^{\tilde{K}}_{\Ocal_K}$. 
Hence ${\bf r}^o(V)=0$ only if $V=0$. $\Box$
\end{proof}

\subsection{Admissibility}

Let $\pi_{\Ocal'}^{G'}$ be a discrete series representation of $G$ attached to a regular admissible elliptic orbit 
$\Ocal'\subset (\ggot')^*$.

We denote $\As(\Ocal')\subset (\ggot')^*$ the asymptotic support of the coadjoint orbit $\Ocal'$: by definition 
$\xi\in \As(\Ocal')$ if $\xi=\lim_{n\to \infty} t_n \xi_n$ with $\xi_n\in \Ocal'$ and $(t_n)$ is a sequence of 
positive number tending to $0$.

We consider here a closed connected semi-simple Lie subgroup $G\subset G'$. We choose maximal compact 
subgroups $K\subset G$ and $K'\subset G'$ such that $K\subset K'$. We denote 
$\kgot^\perp\subset (\kgot')^*$ the orthogonal (for the duality) of 
$\kgot\subset \kgot'$.

The moment map relative to the $G$-action on $\Ocal'$ is by definition the map $\Phi_G:\Ocal'\to\ggot^*$ which is the composition of the 
inclusion $\Ocal'\croc (\ggot')^*$ with the projection $(\ggot')^*\to \ggot^*$. We use also the moment map 
$\Phi_K:\Ocal'\to \kgot^*$ which the composition of $\Phi_G$ with the projection $\ggot^*\to\kgot^*$.

Let $\p_{\kgot',\ggot'} :(\ggot')^*\to (\kgot')^*$ be the canonical projection. The main objective of this section is the proof of the 
following result that refines Proposition \ref{prop:Duflo-Vargas}. 
\begin{theo}\label{theo:equivalence}
The following facts are equivalent~:
\begin{enumerate}
\item The representation $\pi_{\Ocal'}^{G'}$ is $G$-admissible.
\item The moment map $\Phi_G:\Ocal'\to \ggot^*$ is proper.
\item $\p_{\kgot',\ggot'}\left(\As(\Ocal')\right)\cap\kgot^\perp=\{0\}$.
\end{enumerate}
\end{theo}

Theorem \ref{theo:equivalence} is a consequence of different equivalences. We start with the following result that is proved in 
\cite{Duflo-Vargas2007,pep-jems}.
\begin{lem}
The map $\Phi_G:\Ocal'\to \ggot^*$ is proper if and only if the map $\Phi_K:\Ocal'\to \kgot^*$ is proper.
\end{lem}

We have the same kind of equivalence for the admissibility.

\begin{lem}
The representation $\pi_{\Ocal'}^{G'}$ is $G$-admissible if and only if it is $K$-admissible. 
\end{lem}
\begin{proof}
The fact that $K$-admissibilty implies $G$-admissibility is proved by T. Kobayashi in \cite{Toshi-Inventiones94}. 
The opposite implication is a consequence of the first point of Proposition \ref{prop:restriction}.
\end{proof}

At this stage, the proof of Theorem \ref{theo:equivalence} is complete if we show that 
the following facts are equivalent~:
\begin{enumerate}
\item[(a)] The representation $\pi_{\Ocal'}^{G'}$ is $K$-admissible.
\item[(b)] The moment map $\Phi_K:\Ocal'\to \kgot^*$ is proper.
\item[(c)] $\p_{\kgot',\ggot'}\left(\As(\Ocal')\right)\cap\kgot^\perp=\{0\}$.
\end{enumerate}

We start by proving the equivalence $(b)\Longleftrightarrow (c)$.

\begin{prop}[\cite{pep-jems}]\label{prop:proper-moment-map}
The map $\Phi_K:\Ocal'\to \kgot^*$ is proper if and only 
$$
\p_{\kgot',\ggot'}\left(\As(\Ocal')\right)\cap\kgot^\perp=\{0\}.
$$
\end{prop}
\begin{proof}
The moment map $\Phi_{K'}:\Ocal'\to (\kgot')^*$ relative to the action of $K'$ on $\Ocal'$ is  a proper map that 
 corresponds to the 
restriction of the projection $\p_{\kgot',\ggot'}$ to $\Ocal'$.

Let $T'$ be a maximal torus in $K'$ and let $(\tgot')^*_+\subset (\tgot')^*$ be a Weyl chamber. The convexity theorem 
\cite{Kirwan-Invent84,L-M-T-W}
tells us that $\Delta_{K'}(\Ocal')=\p_{\kgot',\ggot'}(\Ocal')\cap (\tgot')^*_+$ is a closed convex polyedral subset. 
We have proved in \cite{pep-jems}[Proposition 2.10], that $\Phi_K:\Ocal'\to \kgot^*$ is proper if and only 
$$
K'\cdot \As(\Delta_{K'}(\Ocal'))\cap\kgot^\perp=\{0\}.
$$
A small computation shows that $K'\cdot \As(\Delta_{K'}(\Ocal'))=\p_{\kgot',\ggot'}\left(\As(\Ocal')\right)$ since 
$K'\cdot\Delta_{K'}(\Ocal')=\p_{\kgot',\ggot'}\left(\Ocal'\right)$. The proof of Proposition \ref{prop:proper-moment-map} 
is completed. $\Box$
\end{proof}

\medskip

We denote $\AS_{K'}(\pi_{\Ocal'}^{G'})\subset(\kgot')^*$ the asymptotic support of the following subset of $(\kgot')^*$:
$$
\{\Pcal'\in \widehat{K'}, [\pi_{\Pcal'}^{K'}:\pi_{\Ocal'}^{G'}]\neq 0\}.
$$

The following important fact is proved by T. Kobayashi (see Section 6.3 in \cite{Toshi-Progress-Math}).
\begin{prop}\label{prop:toshi-admissible}
The representation $\pi_{\Ocal'}^{G'}$ is $K$-admissible if and only if 
$$
\AS_{K'}(\pi_{\Ocal'}^{G'})\cap \kgot^\perp=\{0\}.
$$
\end{prop}

\medskip

We will use also the following result proved by Barbasch and Vogan  (see Propositions 3.5 and 3.6 in \cite{Barbasch-Vogan}).
\begin{prop}\label{prop:barbasch-vogan}
Let $\pi_{\Ocal'}^{G'}$ be a representation of the discrete series of $G'$ attached to the regular admissible elliptic orbit $\Ocal'$. 
We have 
$$
\AS_{K'}(\pi_{\Ocal'}^{G'})=\p_{\kgot',\ggot'}\left(\As(\Ocal')\right).
$$
\end{prop}

Propositions \ref{prop:toshi-admissible} and \ref{prop:barbasch-vogan} give the equivalence 
$(a)\Longleftrightarrow (c)$. The proof of Theorem \ref{theo:equivalence} is completed. $\Box$

\medskip

In fact Barbasch and Vogan  proved also in \cite{Barbasch-Vogan} that the set $\As(\Ocal')$ does not depends on $\Ocal'$ but only 
on the chamber $\Ccal'$ such that $\Ocal'\in \widehat{G'_d}(\Ccal')$. We obtain the following corollary.

\begin{coro}
The $G$-admissibility of a discrete series representation $\pi^{G'}_{\Ocal'}$ does not depends on $\Ocal'$ but only 
on the chamber $\Ccal'$ such that $\Ocal'\in \widehat{G'_d}(\Ccal')$.
\end{coro}

\section{$\Spinc$ quantization of compact Hamiltonian manifolds}

\subsection{$\Spinc$ structures}
Let $N$ be an even dimensional Riemannian manifold, and let $\mathrm{Cl}(N)$ be its Clifford algebra bundle. 
A complex vector bundle $\Ecal\to N$ is a $\mathrm{Cl}(N)$-module if there is a bundle algebra morphism
$\clif_\Ecal : \mathrm{Cl}(N)\longrightarrow \End(\Ecal)$.

\begin{defi}
Let $\Scal\to M$ be  a $\mathrm{Cl}(N)$-module  such that the  map  $\clif_\Scal$ induces an isomorphism
$\mathrm{Cl}(N)\otimes_\Rbb \Cbb\longrightarrow \End(\Scal)$. Then we say that  
$\Scal$ is a $\spinc$-bundle for $N$.
\end{defi}

\begin{defi}\label{def:determinantLine}
The determinant line bundle of a $\spinc$-bundle $\Scal$ on $N$ is the line bundle $\det(\Scal)\to M$ defined by the relation
$$
\det(\Scal):=\hom_{\mathrm{Cl}(N)}(\overline{\mathcal{S}},\mathcal{S})
$$
where $\overline{\mathcal{S}}$ is the $\mathrm{Cl}(N)$-module with opposite complex structure. 
\end{defi}

Basic examples of $\spinc$-bundles are those coming from manifolds $N$ equipped with an almost complex structure $J$. 
We consider the tangent bundle $\T N$ as a complex vector bundle and we define
$$
\Scal_J:=\bigwedge_\Cbb \T N.
$$
It is not difficult to see that $\Scal_J$ is a $\spinc$-bundle on $N$ with determinant line bundle 
$\det(\Scal_J)=\bigwedge_\Cbb^{max} \T N$. If $L$ is a complex line bundle on $N$, then $\Scal_J\otimes L$ 
is another $\spinc$-bundle with 
determinant line bundle equal to $\bigwedge_\Cbb^{max} \T N\otimes L^{\otimes 2}$.

\subsection{$\Spinc$-prequantization}

In this section $G$ is a semi-simple connected real Lie group.

Let $M$ be an Hamiltonian $G$-manifold with symplectic form $\Omega$ 
and moment map $\Phi_G: M\to \ggot^{*}$ characterized by the relation 
\begin{equation}\label{eq:hamiltonian-action}
    \iota(X_M)\omega= -d\langle\Phi_G,X\rangle,\quad X\in\ggot, 
\end{equation}
where $X_M(m):=\frac{d}{dt}\vert_{t=0} e^{-tX}\cdot m$ is the vector field on $M$ generated by $X\in \ggot$.

In the Kostant-Souriau framework \cite{Kostant70,Souriau}, a $G$-equivariant Hermitian line bundle $L_\Omega$ with an 
invariant Hermitian connection $\nabla$ is a prequantum line bundle over $(M,\Omega,\Phi_G)$ if 
\begin{equation}\label{eq:kostant-L}
    \Lcal(X)-\nabla_{X_M}=i\langle\Phi_G,X\rangle\quad \mathrm{and} \quad
    \nabla^2= -i\Omega,
\end{equation}
for every $X\in\ggot$. Here $\Lcal(X)$ is the infinitesimal action of $X\in\kgot$ on the sections 
of $L_\Omega\to M$. The data $(L_\Omega,\nabla)$ is also called a Kostant-Souriau line bundle. 


\begin{defi}[\cite{pep-JSG}]\label{prop:spin-prequantized}
A $G$-Hamiltonian manifold $(M,\Omega,\Phi_G)$ is \break $\spinc$ prequantized if there exists an equivariant 
$\spinc$ bundle $\Scal$ such that its determinant line bundle $\det(\Scal)$ is a prequantum line bundle over \break
$(M,2\Omega,2\Phi_G)$.
\end{defi}

\medskip

Consider the case of a regular elliptic coadjoint orbit $\Ocal=G\lambda$: here $\lambda\in\tgot^*$ has a stabilizer subgroup equal to $T$. The tangent space $\T_\lambda\Ocal\simeq \ggot/\tgot$ is an even dimensional Euclidean space, 
equipped with a linear action of $T$ and an $T$-invariant antisymmetric endomorphism\footnote{Here we see $\lambda$ has an element of $\tgot$, through the identification $\ggot^*\simeq \ggot$.} ${\rm ad}(\lambda)$. 
Let $J_\lambda:= {\rm ad}(\lambda) (-{\rm ad}(\lambda)^2)^{-1/2}$ be the corresponding $T$-invariant 
complex structure on $\ggot/\tgot$: we denote $V$ the corresponding $T$-module. It defines an integrable $G$-invariant complex structure on $\Ocal\simeq G/T$.

As we have explained in the previous section, the complex structure on $\Ocal$ defines the $\spinc$-bundle
$\Scal_o:=\bigwedge_\Cbb \T \Ocal$ with determinant line bundle
$$
\det(\Scal_o)={\bigwedge}_\Cbb^{\max} \T \Ocal\simeq G\times_T \bigwedge_\Cbb^{\max} V.
$$
A small computation gives that the differential of the $T$-character 
$\bigwedge_\Cbb^{\max} V$ is equal to $i$ times $2\rho(\lambda)$. In other words, 
$\bigwedge_\Cbb^{\max} V=\Cbb_{2\rho(\lambda)}$.

In the next Lemma we see that for the regular elliptic orbits, the notion of {\em admissible} orbits is equivalent to 
the notion of $\spinc$-{\em prequantized} orbits.

\begin{lem}\label{lem:admissible=prequantized}
Let $\Ocal=G\lambda$ be a regular elliptic coadjoint orbit. Then $\Ocal$ is $\spinc$-prequantized if and only if $\lambda-\rho(\lambda)\in \Lambda$.
\end{lem}
\begin{proof} Any $G$-equivariant $\spinc$-bundle on $\Ocal$ is of the form 
$\Scal_\phi=\Scal_o\otimes L_\phi$ where $L_\phi=G\times_T\Cbb_\phi$ is a line bundle associated 
to a character $e^X\mapsto e^{i\langle\phi,X\rangle}$ of the group $T$. Then we have 
$$
\det(\Scal_\phi)=\det(\Scal_o)\otimes L_\phi^{\otimes 2}= G\times_T\Cbb_{2\phi+2\rho(\lambda)}.
$$

By $G$-invariance we know that the only Kostant-Souriau line bundle on $(G\lambda,2\Omega_{G\lambda})$ is the line bundle
$G\times_T \Cbb_{2\lambda}$. Finally we see that $G\lambda$ is $\spinc$-prequantized by $\Scal_\phi$ if and only if 
$\phi=\lambda-\rho(\lambda)$. $\Box$
\end{proof}

\medskip

If $\Ocal$ is a regular admissible elliptic coadjoint orbit, we denote $\Scal_\Ocal:=\Scal_o\otimes L_{\lambda-\rho(\lambda)}$ 
the corresponding $\spinc$ bundle. Here we use the grading  
$\Scal_\Ocal=\Scal^+_\Ocal\oplus\Scal_\Ocal^-$ induced by the symplectic orientation.

\subsection{$\Spinc$ quantization of compact manifolds}

Let us consider a compact Hamiltonian $K$-manifold $(M,\Omega,\Phi_K)$ which is $\spinc$-prequantized by a
$\spinc$-bundle $\Scal$. The (symplectic) orientation induces a decomposition $\Scal=\Scal^+\oplus \Scal^-$, and 
the corresponding $\spinc$ Dirac operator is a first order elliptic operator \cite{B-G-V} 
$$\Dcal_\Scal: \Gamma(M,\Scal^+)\to \Gamma(M,\Scal^-).
$$  
Its principal symbol is the bundle map  $\sigma(M,\Scal)\in \Gamma(\T^*M,\hom(p^*\Scal^+,p^*\Scal^-))$ defined by the relation
$$
\sigma(M,\Scal)(m,\nu)=\clif_{\Scal\vert_m}(\tilde{\nu}): \Scal\vert_m^+\longrightarrow\Scal\vert_m^-.
$$
Here $\nu\in\T^* M\mapsto \tilde{\nu}\in \T M$ is the identification defined by an invariant Riemannian structure.

\begin{defi} \label{defi:spinc-quantization}
The $\spinc$ quantization of a compact Hamiltonian $K$-manifold $(M,\Omega,\Phi_K)$ is the equivariant index of 
the elliptic operator  $\Dcal_\Scal$ and is denoted
$$
\QS_K(M)\in R(K).
$$
\end{defi}


\subsection{Quantization commutes with reduction}

Now we will explain how the multiplicities of $\QS_K(M)\in R(K)$ can be computed geometrically. 

Recall that the dual $\what{K}$ is parametrized by the regular admissible coadjoint orbits. They are those of the form
$\Pcal=K\mu$ where $\mu-\rho_c\in \Lambda\cap\tgot^*_+$. After Lemma \ref{lem:admissible=prequantized}, we know 
that any regular admissible coadjoint orbit $\Pcal\in\wK$ is $\spinc$-prequantized by a $\spinc$ bundle $\Scal_{\Pcal}$ 
and a small computation shows that $\QS_K(\Pcal)=\pi_\Pcal^K$ (see \cite{pep-vergne:magic}). 

For any $\Pcal\in\what{K}$, we define the symplectic reduced space
$$
M\slash\!\!\slash \Pcal:=\Phi_K^{-1}(\Pcal)/K.
$$
If $M\slash\!\!\slash \Pcal\neq \emptyset$, then any $m\in \Phi_K^{-1}(\Pcal)$ has abelian infinitesimal stabilizer. It implies then that the generic infinitesimal stabilizer for the $K$-action on $M$ is {\em abelian}.

Let us explain how we can quantize these symplectic reduced spaces (for more details see \cite{pep-ENS, pep-JSG,pep-vergne-acta}).

\begin{prop}\label{prop:QS-definition}
Suppose that the generic infinitesimal stabilizer for the $K$-action on $M$ is abelian.

$\bullet$ If $\Pcal\in\what{K}$ belongs to the set of regular values of $\Phi_K:M\to \kgot^*$, then \break 
$M\slash\!\!\slash \Pcal$ is a compact symplectic orbifold which is $\spinc$-prequantized. 
We denote $\QS(M\slash\!\!\slash \Pcal)\in \Zbb$ the index 
of the corresponding $\spinc$ Dirac operator \cite{Kawasaki81}.

$\bullet$ In general, if $\Pcal=K\lambda$ with $\lambda\in \tgot^*$, we consider the orbits $\Pcal_\epsilon=K(\lambda+\epsilon)$ 
for generic small elements $\epsilon\in\tgot^*$ so that $M\slash\!\!\slash \Pcal_\epsilon$ is a compact symplectic orbifold with a 
peculiar $\spinc$-structure. Let $\QS(M\slash\!\!\slash \Pcal_\epsilon)\in \Zbb$ be the index of the corresponding $\spinc$ Dirac operator. 
The crucial fact is that the quantity $\QS(M\slash\!\!\slash \Pcal_\epsilon)$ does not depends on the 
choice of generic and small enough $\epsilon$. Then we take
$$
\QS(M\slash\!\!\slash \Pcal):=\QS(M\slash\!\!\slash \Pcal_\epsilon)
$$
 for generic and small enough $\epsilon$.
\end{prop}

The following theorem is proved in \cite{pep-ENS}.

\begin{theo}\label{theo:QR-spinc-compact}
Let $(M,\Omega,\Phi_K)$ be  a $\spinc$-prequantized compact Hamiltonian $K$-manifold. Suppose that the generic infinitesimal stabilizer for the $K$-action on $M$ is {\em abelian}. Then the following relation holds in $R(K)$:
\begin{equation}\label{eq:QR-spinc-compact}
\QS_K(M)=\sum_{\Pcal\in\what{K}}\QS(M\slash\!\!\slash \Pcal)\, \pi^K_\Pcal.
\end{equation}
\end{theo}

\begin{rem}
Identity \ref{eq:QR-spinc-compact} admits generalisations when we do not have conditions on the generic stabilizer \cite{pep-JSG} and also when we allow the $2$-form $\Omega$ to be degenerate \cite{pep-vergne-acta}. In this article, we do not need such generalizations.
\end{rem}

For $\Pcal\in \what{K}$, we denote $\Pcal^-$ the coadjoint orbit with $\Pcal$ with opposite symplectic structure. 
The corresponding $\spinc$ bundle is $\Scal_{\Pcal^-}$. It is not difficult to see that $\QS_K(\Pcal^-)=(\pi_\Pcal^K)^*$ 
(see \cite{pep-vergne:magic}). 
The shifting trick tell us then that the multiplicity of $\pi^K_\Pcal$ in $\QS_K(M)$ is equal to 
$[\QS_K(M\times \Pcal^-)]^K$. If we suppose furthermore that the generic infinitesimal stabilizer is abelian we obtain the useful relation
\begin{equation}\label{eq:shift}
\QS(M\slash\!\!\slash \Pcal):=\left[\QS_K(M\times \Pcal^-)\right]^K.
\end{equation}

Let $\gamma$ that belongs to the center of $K$: it acts trivially on the orbits $\Pcal\in \what{K}$. 
Suppose now that $\gamma$ acts also trivially on the manifolds $M$. We are interested by the action 
of $\gamma$ on the fibers of the $\spinc$-bundle $\Scal\boxtimes \Scal_{\Pcal^-}$. We denote 
$[\Scal\boxtimes \Scal_{\Pcal^-}]^\gamma$ the subbundle where $\gamma$ acts trivially.

\begin{lem}\label{lem:gamma-compact}
If $[\Scal\boxtimes \Scal_{\Pcal^-}]^\gamma=0$ then $\QS(M\slash\!\!\slash \Pcal)=0$.
\end{lem}

\begin{proof} 
Let $D$ be the Dirac operator on $M\times \Pcal^-$ associated to the $\spinc$ bundle 
$\Scal\boxtimes \Scal_{\Pcal^-}$. Then 
$$
\left[\QS_K(M\times \Pcal^-)\right]^K=[{\rm ker }(D)]^K - [{\rm coker}(D)]^K.
$$
Obviously $[{\rm ker }(D)]^K\subset [{\rm ker }(D)]^\gamma$ and $[{\rm ker }(D)]^\gamma$ is contained 
in the set of smooth section of the bundle $[\Scal\boxtimes \Scal_{\Pcal^-}]^\gamma$. The same result holds for 
$[{\rm coker}(D)]^K$. Finally, if $[\Scal\boxtimes \Scal_{\Pcal^-}]^\gamma=0$, then 
$[{\rm ker }(D)]^K$ and  $[{\rm coker}(D)]^K$ are reduced to $0$. $\Box$

\end{proof}

\section{$\Spinc$ quantization of non-compact Hamiltonian manifolds}

In this section our Hamiltonian $K$-manifold $(M,\Omega,\Phi_K)$ is not necessarily compact, but the moment map $\Phi_K$ is 
supposed to be proper. We assume that $(M,\Omega,\Phi_K)$ is $\spinc$-prequantized by a $\spinc$-bundle $\Scal$.

In the next section, we will explain how to quantize the data $(M,\Omega,\Phi_K,\Scal)$.

\subsection{Formal geometric quantization : definition}

We choose an invariant scalar product in $\kgot^*$ that provides an identification $\kgot\simeq\kgot^*$.

\begin{defi}\label{defi:kir}
$\bullet$ The {\em Kirwan vector field} associated to $\Phi_K$ is defined by
\begin{equation}\label{eq-kappa}
    \kappa(m)= -\Phi_K(m)\cdot m, \quad m\in M.
\end{equation}
\end{defi}

We denote by $Z_M$ the set of zeroes of $\kappa$.  It is not difficult to see that $Z_M$ corresponds to the set of 
critical points of the function $\|\Phi_K\|^2: M\to\Rbb$.

The set $Z_M$, which is not necessarily smooth, admits the following description. Choose a Weyl chamber 
$\mathfrak{t}^*_+\subset \mathfrak{t}^*$ in the dual 
of the Lie algebra of a maximal torus $T$ of $K$. We see that 
\begin{equation}\label{eq=Z-Phi-beta}
Z_M=\coprod_{\beta\in\Bcal} Z_\beta
\end{equation}
where $Z_\beta$ corresponds to the compact set $K(M^\beta\cap\Phi_K^{-1}(\beta))$, and $\Bcal=\Phi_K(Z_M)\cap \tgot^*_+$.  
The properness of $\Phi_K$ insures that for any compact subset $C\subset \tgot^*$ the intersection $\Bcal\cap C$ is finite.

The principal symbol of the Dirac operator $D_\Scal$ is the bundle map
$\sigma(M,\Scal)\in \Gamma(\T^* M, \hom(\Scal^+,\Scal^-))$ defined by the Clifford action
$$\sigma(M,\Scal)(m,\nu)=\clif_{m}(\tilde{\nu}): \Scal\vert_m^+\to \Scal\vert_m^-.$$
where $\nu\in \T^* M\simeq \tilde{\nu}\in \T M$ is an identification associated to an invariant Riemannian metric on $M$.
\begin{defi}\label{def:pushed-sigma}
The symbol  $\sigma(M,\Scal,\Phi_K)$ shifted by the vector field $\kappa$ is the
symbol on $M$ defined by
$$
\sigma(M,\Scal,\Phi_K)(m,\nu)=\sigma(M,\Scal)(m,\tilde{\nu}-\kappa(m))
$$
for any $(m,\nu)\in\T^* M$.
\end{defi}

For any $K$-invariant open subset $\Ucal\subset M$ such that $\Ucal\cap Z_M$ is compact in $M$, we see that the restriction
$\sigma(M,\Scal,\Phi_K)\vert_\Ucal$ is a transversally elliptic symbol on $\Ucal$, and so its equivariant index is a well defined element in
$\wR(K)$ (see \cite{Atiyah74,pep-vergne:witten}).

Thus we can define the following localized equivariant indices.

\begin{defi}\label{def:indice-localise}
\begin{itemize}
\item A closed invariant subset $Z\subset Z_M$ is called a {\em component} of $Z_M$ if it is a union of connected components of 
$Z_M$.
\item  If $Z$ is a {\em compact component} of $Z_M$, we denote by
$$
\QS_K(M,Z)\ \in\ \widehat{R}(K)
$$
the equivariant index of $\sigma(M,\Scal,\Phi_K)\vert_\Ucal$ where $\Ucal$ is an invariant neighbourhood of $Z$
so that $\Ucal\cap Z_M=Z$.
\end{itemize}
\end{defi}

By definition, $Z=\emptyset$ is a component of $Z_M$ and $\QS_K(M,\emptyset)=0$. For any $\beta\in\Bcal$, $Z_\beta$ is a 
compact component of  $Z_M$. 

When the manifold $M$ is compact, the set $\Bcal$ is finite and we have the decomposition
$$
\QS_K(M)=\sum_{\beta\in\Bcal}\QS_K(M,Z_\beta) \quad \in\, \wR(K).
$$
See \cite{pep-RR,pep-vergne:witten}. When the manifold $M$ is not compact, but the moment map $\Phi_K$ is proper, we can define
$$
\QSf_K(M):=\sum_{\beta\in\Bcal}\QS_K(M,Z_\beta)\ \in\ \wR(K).
$$
The sum of the right hand side is not necessarily finite but it converges in $\wR(K)$ (see \cite{pep-pacific,Ma-Zhang14,Hochs-Song-duke}). 

\begin{defi}\label{def:formal-quant-spin}
We call $\QSf_K(M)\in \widehat{R}(K)$ the $\spinc$ formal geometric quantization of the Hamiltonian manifold $(M,\Omega,\Phi_K)$.
\end{defi}

We end up this section with the example of the coadjoint orbits that parametrize the discrete series representations. 
We have seen in Lemma \ref{lem:admissible=prequantized} that any $\Ocal\in\wG_d$ is $\spinc$-prequantized. 
Moreover, if we look at the $K$-action on $\Ocal$, we know also that the moment map 
$\Phi_K:\Ocal\to \kgot^*$ is proper. The element $\QSf_K(\Ocal)\in\wR(K)$ is then well-defined.

The following result can be understood as a geometric interpretation of the Blattner formula.

\begin{prop}[\cite{pep-ENS}]\label{prop:Q-K-O}
For any $\Ocal\in\wG_d$ we have the following equality in $\wR(K)$:
$$
\QSf_K(\Ocal)=\pi_\Ocal^G\vert_K.
$$
\end{prop} 

\medskip

\subsection{Formal geometric quantization: main properties}

In this section, we recall two important functorial properties of the formal geometric quantization process $\QSf$.

We start with the following result of Hochs and Song.

\begin{theo}[\cite{Hochs-Song-duke}]\label{theo:hochs-song}
Let $(M,\Omega,\Phi_K)$ be a $\spinc$ prequantized Hamiltonian $K$-manifold. Assume that the moment map $\Phi_K$ is 
 proper and that the generic infinitesimal stabilizer for the $K$-action on 
$M$ is {\em abelian}. Then the following relation holds in 
$\wR(K)$:
\begin{equation}\label{eq:QR-spinc-non-compact}
\QSf_K(M)=\sum_{\Pcal\in\what{K}}\QS(M\slash\!\!\slash \Pcal)\, \pi^K_\Pcal.
\end{equation}
\end{theo}

\begin{rem}
Identity (\ref{eq:QR-spinc-non-compact}) admits generalizations when we do not have conditions on the generic stabilizer and also when we allow the $2$-form $\Omega$ to be degenerate (see \cite{Hochs-Song-duke}).
\end{rem}

Like in the compact setting, consider an element $\gamma$ belonging to the center of $K$ that acts 
trivially on the manifold $M$. Let $\Pcal\in \wK$ and let $\Pcal^-$ be the orbit $\Pcal$ with opposite symplectic structure.
We are interested by the action of $\gamma$ on the fibers of the $\spinc$-bundle $\Scal\boxtimes \Scal_{\Pcal^-}$. 
We denote $[\Scal\boxtimes \Scal_{\Pcal^-}]^\gamma$ the subbundle where $\gamma$ acts trivially.

Lemma \ref{lem:gamma-compact} extends to the non-compact setting.
\begin{lem}\label{lem:gamma-non-compact}
If $[\Scal\boxtimes \Scal_{\Pcal^-}]^\gamma=0$ then $\QS(M\slash\!\!\slash \Pcal)=0$.
\end{lem}

\begin{proof} The multiplicative property proved by Hochs and Song \cite{Hochs-Song-duke} tells us that the shifting trick still 
holds in the non compact setting: the multiplicity of $\pi^K_\Pcal$ in $\QSf_K(M)$ is equal to 
$[\QSf_K(M\times \Pcal^-)]^K$. If we suppose furthermore that the generic infinitesimal stabilizer is abelian we obtain 
\begin{eqnarray*}\label{eq:shift-non-compact}
\QS(M\slash\!\!\slash \Pcal)&=&\left[\QSf_K(M\times \Pcal^-)\right]^K\\
&=&\left[\QS_K(M\times \Pcal^-, Z_0)\right]^K
\end{eqnarray*}
where $Z_0\subset M\times \Pcal^-$ is the compact set $\{(m,\xi)\in M\times \Pcal^-,\, \Phi_K(m)=\xi\}$.

The quantity $\QS_K(M\times \Pcal^-, Z_0)\in \wR(K)$ is computed as an index of a $K$-transversally elliptic operator $D_0$ 
acting on the sections of $\Scal\boxtimes \Scal_{\Pcal^-}$. The argument used in the compact setting still work 
(see Lemma 1.3 in \cite{pep-vergne:witten}):   if $[\Scal\boxtimes \Scal_{\Pcal^-}]^\gamma=0$ then 
$[{\rm ker }(D_0)]^K$ and  $[{\rm coker}(D_0)]^K$ are reduced to $0$. $\Box$ 
\end{proof}

\medskip

Another important property of the formal geometric quantization procedure is the functoriality relatively to restriction to subgroup. 
Let $H\subset K$ be a closed connected subgroup. We denote $\Phi_H:M\to\hgot^*$ the moment map relative to the 
$H$-action: it is equal to the composition of $\Phi_K$ with the projection $\kgot^*\to\hgot^*$.

\begin{theo}[\cite{pep-formal3}]\label{theo:pep-formal-3}
Let $(M,\Omega,\Phi_K)$ be a $\spinc$ prequantized Hamiltonian $K$-manifold. Assume that the moment map $\Phi_H$ is a proper. Then the element $\QSf_K(M)\in\wR(K)$ is $H$-admissible and we have 
$$
\QSf_K(M)\vert_H=\QSf_H(M).
$$
\end{theo}

If we apply the previous Theorem to the $\spinc$-prequantized coadjoint orbits $\Ocal\in\wG_d$, we obtain the following extension of Proposition \ref{prop:Q-K-O}. This result was obtained by other means by Duflo-Vergne  
\cite{Duflo-Vergne2011}.

\begin{coro}\label{coro:restriction-pi-O-G-H}
Let $\Ocal\in\wG_d$, and $H\subset K$ a closed connected subgroup such that 
$\Phi_H:\Ocal\to \hgot^*$ is proper. Then $\pi^G_\Ocal$ is $H$-admissible and 
$$
\QSf_H(\Ocal)=\pi^G_\Ocal\vert_H.
$$
\end{coro}

\section{$\Spinc$ quantization of $G$-Hamiltonian manifolds}
 
In this section $G$ denotes a connected semi-simple Lie group, and we consider a symplectic manifold $(M,\Omega)$ equipped with an Hamiltonian action of $G$: we denote $\Phi_G:M\to \ggot^*$ the corresponding moment map.

\subsection{Proper$^2$ Hamiltonian $G$-manifolds}\label{sec:proper-2}

In this section we suppose that:
\begin{enumerate}
\item the moment map $\Phi_G$ is {\em proper},
\item the $G$-action on $M$ is {\em proper}.
\end{enumerate}
For simplicity, we says that $(M,\Omega,\Phi_G)$ is a proper$^2$ Hamiltonian $G$-manifold.

 Following Weinstein \cite{Weinstein01}, we consider the $G$-invariant open subset 
\begin{equation}\label{eq:def-se}
\ggot^*_{se}=\{\xi\in\ggot^*\,\vert\, G_\xi \ \mathrm{is\  compact} \}
\end{equation}
of strongly elliptic elements.  It is non-empty if and only if  the groups $G$ and $K$ have the same rank  : real semi-simple Lie groups with this property are the ones admitting discrete series. If we denote 
$\tgot^*_{se}:=\ggot^*_{se}\cap\tgot^*$, we see that $\ggot^*_{se}=G\cdot\tgot^*_{se}$. In other words, any coadjoint 
orbit  contained in $\ggot^*_{se}$ is elliptic.

\medskip

{

First we recall the geometric properties associated to proper$^2$ Hamiltonian $G$-manifolds. 
We denote $K$ a maximal compact subgroup of $G$ and we denote $\Phi_K:M\to\kgot^*$ the moment map relative to the 
$K$-action on $(M,\Omega)$.

\begin{prop}[\cite{pep-jems}]
\label{prop:proper-2} 
Let $(M,\Omega,\Phi_G)$ be a proper$^2$ Hamiltonian $G$-manifold.

Then:
\begin{enumerate}
\item the map $\Phi_K$ is proper,
\item the set $\ggot^*_{se}$ is non-empty,
\item the image of $\Phi_G$ is contained in $\ggot^*_{se}$,
\item the set $N:=\Phi_G^{-1}(\kgot^*)$ is a smooth $K$-submanifold of $M$,
\item the restriction of $\Omega$ on $N$ defines a symplectic form $\Omega_N$,
\item the map $[g,n]\mapsto gn$ defines a diffeomorphism $G\times_K N\simeq M$.
\end{enumerate}
\end{prop}

\medskip

Let $T$ be a maximal torus in $K$, and let $\tgot^*_+$ be a Weyl chamber. Since any coadjoint orbit in $\ggot^*_{se}$ 
is elliptic, the coadjoint orbits belonging to the image of $\Phi_G:N\to \ggot^*$ are parametrized by the set
\begin{equation}
\Delta_G(M)=\Phi_G(M)\cap \tgot^*_+.
\end{equation}

We remark that $\tgot^*_+\cap \ggot^*_{se}$ is equal to 
$(\tgot^*_+)_{se}:=\{\xi\in \tgot^*_+, (\xi,\alpha)\neq 0,\ \forall \alpha\in\Rgot_n\}$.  
The connected component $(\tgot^*_+)_{se}$ are called chambers and if 
$\Ccal$ is a chamber, we denote $\wG_d(\Ccal)$ the set of regular admissible elliptic orbits intersecting $\Ccal$ 
(see Definition \ref{defi:chamber}).

The following fact was first noticed by Weinstein \cite{Weinstein01}.

\begin{prop}\label{prop:G-convexity}
$\Delta_G(M)$ is a convex polyhedral set contained in a unique chamber $\Ccal_M\subset (\tgot^*_+)_{se}$. 
\end{prop}

\begin{proof} We denote $\Phi_K^N: N\to \kgot^*$ the restriction of the map $\Phi_G$ on the sub-manifold $N$. 
It corresponds to the moment map relative to the $K$-action  on $(N,\Omega_N)$: notice that $\Phi_K^N$ is a proper map.

The diffeomorphism $G\times_K N\simeq M$ shows that 
the set $\Delta_G(M)$ is equal to $\Delta_K(N):={\rm Image}(\Phi_K^N)\cap \tgot^*_+$, and the Convexity 
Theorem \cite{Kirwan-Invent84,L-M-T-W} asserts that $\Delta_K(N)$ is a convex polyhedral subset of the Weyl chamber. Finally since $\Delta_K(N)$ 
is connected and contained in $(\tgot^*_+)_{se}$, it must belongs to a unique chamber $\Ccal_M$. $\Box$
\end{proof}

\subsection{$\Spinc$-quantization of proper${}^2$ Hamiltonian $G$-manifolds}\label{sec:quantization-proper-2}

Now we assume that our proper$^2$ Hamiltonian $G$-manifold $(M,\Omega,\Phi_G)$ is $\spinc$-prequantized by a 
$G$-equivariant $\spinc$-bundle $\Scal$. 


\medskip

Note that $\pgot$ is even dimensional since the groups $G$ and $K$ have the same rank.  Recall that 
the morphism $K\to {\rm SO}(\pgot)$ lifts to a morphism $\tilde{K}\to  {\rm Spin}(\pgot)$, where 
$\tK\to K$ is either an isomorphism or a two-fold cover (see Section \ref{sec:spinor-representation}). We start with the

\begin{lem}${}$

$\bullet$ The $G$-equivariant $\spinc$ bundle $\Scal$ on $M$ induces a $\tilde{K}$-equivariant $\spinc$ bundle $\Scal_N$ on $N$ such that $\det(\Scal_N)=\det(\Scal)\vert_N$.

$\bullet$ The $\tilde{K}$-Hamiltonian manifold $(N,\Omega_N,\Phi^N_K)$ is $\spinc$-prequantized by $\Scal_N$. 
\end{lem}
\begin{proof}  By definition we have $\T M\vert_N=\pgot \oplus \T N$. The manifolds $M$ and $N$ are oriented by their symplectic forms. The vector space $\pgot$ inherits an orientation $o(\pgot,N)$ satisfying the relation
$o(M)=o(\pgot,N) o(N)$. The orientation $o(\pgot,N)$ can be computed also as follows: takes any $\xi\in 
{\rm Image}(\Phi_K^N)$, then $o(\pgot,N)=o(\xi)$ (see Example \ref{exam:orientation}).

Let $\Scal_\pgot$ be the spinor representation that we see as a $\tilde{K}$-module. The orientation $o(\pgot):=o(\pgot,N)$ determines a decomposition 
$\Scal_\pgot=\Scal_\pgot^{+,o(\pgot)}\oplus \Scal_\pgot^{-,o(\pgot)}$ and we denote
$$
\Scal_\pgot^{o(\pgot)}:= \Scal_\pgot^{+,o(\pgot)}\ominus \Scal_\pgot^{-,o(\pgot)}\in  R(\tilde{K}).
$$
Let $\Scal_N$ be the unique $\spinc$-bundle, $\tilde{K}$-equivariant on $N$ defined by the relation
\begin{equation}\label{eq:spin-divise}
\Scal\vert_N=\Scal_\pgot^{o(\pgot)}\boxtimes \Scal_N.
\end{equation}
Since $\det(\Scal_\pgot^{o(\pgot)})$ is trivial (as $\tilde{K}$-module), we have the relation 
$\det(\Scal_N)=\det(\Scal)\vert_N$ that implies the second point. $\Box$
\end{proof}

\medskip

For $\Ocal\in \wG_d$, we consider the symplectic reduced space
$$
M\slash\!\!\slash \Ocal:=\Phi_G^{-1}(\Ocal)/G.
$$
Notice that $M\slash\!\!\slash \Ocal=\emptyset$ when $\Ocal$ does not belongs to $\wG_d(\Ccal_M)$. Moreover 
the diffeomorphism $G\times_K N\simeq M$ shows that $M\slash\!\!\slash \Ocal$ is equal to the reduced space
$$
N\slash\!\!\slash \Ocal_K := (\Phi_K^N)^{-1}(\Ocal_K)/K.
$$
with $\Ocal_K=\Ocal\cap \kgot^*$. Here $N\slash\!\!\slash \Ocal_K$ should be understood as the symplectic 
reduction of the $\tK$-manifold $N$ relative to the $\tK$-admissible coadjoint orbit $\Ocal_K\in \what{\tilde{K}}$. 
Hence the quantization $\QS(N\slash\!\!\slash \Ocal_K)\in \Zbb$ of the reduced space $N\slash\!\!\slash \Ocal_K$ is well 
defined (see Proposition \ref{prop:QS-definition}).

\begin{defi}
For any $\Ocal\in \wG_d$, we take $\QS(M\slash\!\!\slash \Ocal):=\QS(N\slash\!\!\slash \Ocal_K)$. 
\end{defi}

The main tool to prove Theorem \ref{theo:main} is the comparison of the formal geometric quantization of three different geometric data: we work here in the setting where the $G$-action on $M$ has {\em abelian infinitesimal stabilizers}.
\medskip

\begin{enumerate}
\item The formal geometric quantization of the $G$-action on $(M,\Omega,\Phi_G, \Scal)$ is the
element $\QSf_G(M)\in \wR(G,d)$ defined by the relation
\begin{equation*}\label{eq:quantization-G-action-M}
\QS_G(M):=\sum_{\Ocal\in\wG}\QS(M\slash\!\!\slash \Ocal)\, \pi^G_\Ocal.
\end{equation*}
\item The formal geometric quantization of the $K$-action on $(M,\Omega,\Phi_K,\Scal)$ is the element 
$\QSf_K(M)\in \wR(K)$ (see Definition \ref{def:formal-quant-spin}). As the $K$-action on $M$ has {\em abelian infinitesimal stabilizers}, we have the decomposition
\begin{equation*}\label{eq:quantization-K-action-M}
\QSf_K(M)=\sum_{\Pcal\in\wK}\QS(M\slash\!\!\slash \Pcal)\, \pi^K_\Pcal.
\end{equation*}

\item The formal geometric quantization of the $\tilde{K}$-action on $(N,\Omega_N,\Phi_K^N,\Scal_N)$ is the element 
$\QSf_{\tilde{K}}(N)\in \wR(\tilde{K})$. As the $\tK$-action on $N$ has {\em abelian infinitesimal stabilizers}, we have the decomposition
\begin{equation*}\label{eq:quantization-K-tilde-action}
\QSf_{\tilde{K}}(N)=\sum_{\tilde{\Pcal}\in\what{\tilde{K}}}\QS(N\slash\!\!\slash \tilde{\Pcal})\, \pi^{\tilde{K}}_{\tilde{\Pcal}}.
\end{equation*}
\end{enumerate}

In the next section we explain the link between these three elements.

\subsection{$\Spinc$-quantization: main results}

Let $\Ccal_M\subset \tgot^*_+$ be the chamber containing $\Phi_G(M)\cap\tgot^*_+$. 

\begin{defi} We defines the orientation $o^+$ and $o^-$ on $\pgot$ as follows. Take $\lambda\in \Ccal_M$, then
$o^+:=o(\lambda)$ and $o^-:=o(-\lambda)$ (see Example \ref{exam:orientation}). 
\end{defi} 

We denote $\Scal_\pgot^{o^+},\Scal_\pgot^{o^-}$ the virtual representations of $\tilde{K}$ associated to the spinor 
representation of $\Spin(\pgot)$ and the orientations $o^+$ and $o^-$. We denote $\overline{\Scal_\pgot^{o^+}}$ the $\tilde{K}$-module with opposite complex structure. Remark that $\overline{\Scal_\pgot^{o^+}}\simeq \Scal_\pgot^{o^-}$.

Recall that the map $V\mapsto V\vert_K$ defines a morphism $\wR(G,d)\to \wR(K)$. We have also the morphism
${\bf r}^o=\wR(G,d)\to \wR(\tilde{K})$ defined by ${\bf r}^o(V)=V\vert_K\otimes\Scal_\pgot^o$.

We start with the following

\begin{theo}\label{theo:main-1}
 If the $G$-action on $M$ has {\em abelian infinitesimal stabilizers} then 
\begin{equation}\label{eq:main-1}
{\bf r}^o\left(\QSf_G(M)\right)\,=\,\epsilon^o_M\, \QSf_{\tilde{K}}(N).
\end{equation}
Here $\epsilon^o_M=\pm$ is equal to the ratio between $o$ and $o^-$.
\end{theo}

\begin{proof} If the $G$-action on $M$ has abelian infinitesimal stabilizers, then the 
$\tilde{K}$-action on $N$ has also abelian infinitesimal stabilizers. It implies the following relation:
$$
\QSf_{\tilde{K}}(N)=\sum_{\tilde{\Pcal}\in\what{\tilde{K}}}\QS(N\slash\!\!\slash \tilde{\Pcal})\, \pi^{\tilde{K}}_{\tilde{\Pcal}}
\quad \in \wR(\tilde{K}).
$$
Following the first point of Lemma \ref{lem:K-out}, we consider the following subset 
$\Gamma:= \{\Ocal_K:=\Ocal\cap\kgot^*,\, \Ocal\in\wG_d\}\subset \wK_{out}\subset \what{\tK}$.

Thanks to the second point of Lemma \ref{lem:K-out} we have
\begin{eqnarray*}\label{eq:r-Q-G-M}
{\bf r}^o\left(\QSf_G(M)\right)&=& \sum_{\Ocal\in\wG_d}\QS(M\slash\!\!\slash \Ocal) \, \pi^G_\Ocal\vert_K\otimes \Scal_\pgot^o.\nonumber\\
&=& \epsilon_M^o \sum_{\Ocal\in\wG_d}\QS(N\slash\!\!\slash \Ocal_K) \, \pi^{\tilde{K}}_{\Ocal_K}\\
&=& \epsilon_M^o \sum_{\tilde{\Pcal}\in \Gamma}\QS(N\slash\!\!\slash \tilde{\Pcal}) \, \pi^{\tilde{K}}_{\tilde{\Pcal}} .
\end{eqnarray*}
Identity (\ref{eq:main-1}) is proved if we check that $\QS(N\slash\!\!\slash \tilde{\Pcal})=0$ for any 
$\tilde{\Pcal}\in\widehat{\tilde{K}}$ which does not belong to 
$\Gamma$.

Suppose first that $\tK\simeq K$. In this case we have 
$\widehat{\tilde{K}}=\wK_{out}=\wK$ and a coadjoint orbit $\tilde{P}=K\mu\in\wK$ does not belong to  $\Gamma$ if and only if $\mu$ is not contained in $\ggot^*_{se}$. But the image of $\Phi_G$ is contained in $\ggot^*_{se}$, so 
$N\slash\!\!\slash \tilde{\Pcal}=\emptyset$ and then $\QS(N\slash\!\!\slash \tilde{\Pcal})=0$ if $\tilde{P}\notin\Gamma$.

Suppose now that $\tilde{K}\to K$ is a two-fold cover and let us denote by $\{\pm 1_{\tilde{K}}\}$ the kernel 
of this morphism. Here $\gamma:= -1_{\tilde{K}}$ acts trivially on $N$ and (\ref{eq:spin-divise}) shows that $\gamma$ acts 
by multiplication by $-1$ on the fibers of the $\spinc$ bundle $\Scal_N$. 
The element $\gamma$ acts also trivially on the orbits $\tilde{\Pcal}\in \widehat{\tilde{K}}$:
\begin{itemize}
\item if $\tilde{P}\in\wK_{out}$, then $\gamma$ acts by multiplication by $-1$ on the fibers of the 
$\spinc$ bundle $\Scal_{\tilde{\Pcal}}$,
\item if $\tilde{P}\notin\wK_{out}$,  then $\gamma$ acts trivially on the fibers of the 
$\spinc$ bundle $\Scal_{\tilde{\Pcal}}$.
\end{itemize}

Our considerations show that $[\Scal_N\boxtimes \Scal_{\tilde{\Pcal}^-}]^\gamma=0$ when 
$\tilde{P}\in\widehat{\tilde{K}}\setminus \wK_{out}$. Thanks to Lemma \ref{lem:gamma-non-compact}, it implies 
the vanishing of $\QS(N\slash\!\!\slash \tilde{P})$ for any $\tilde{P}\in\widehat{\tilde{K}}\setminus \wK_{out}$.

Like in the previous case, when $\tilde{P}\in\wK_{out}\setminus \Gamma$, we have $\QS(N\slash\!\!\slash \tilde{P})=0$ because $N\slash\!\!\slash \tilde{\Pcal}=\emptyset$. $\Box$

\end{proof}

\medskip

We compare now the formal geometric quantizations of the $K$-manifolds $M$ and $N$.

\begin{theo}\label{theo:main-2}
We have the following relation
\begin{equation}\label{eq:main-2}
\QSf_K(M)\otimes \overline{\Scal_\pgot^{o^+}}\,=\,\QSf_{\tilde{K}}(N) \quad \in R(\tilde{K}).
\end{equation}
\end{theo}

\medskip

When $M=\Ocal\in\wG_d$ the manifold $N$ is equal to $\Ocal_K:=\Ocal\cap\kgot^*$. We have 
$\QSf_{\tilde{K}}(N)= \pi^{\tilde{K}}_{\Ocal_K}$ and  we know also that $\QSf_K(\Ocal)=\pi_\Ocal^G\vert_K$
(see Proposition \ref{prop:Q-K-O}). Here (\ref{eq:main-2}) becomes
\begin{equation}\label{eq:pi-O-K-times-S}
\pi_\Ocal^G\vert_K\otimes \Scal_\pgot^o =\pm \pi^{\tilde{K}}_{\Ocal_K}
\end{equation}
where the sign $\pm$ is the ratio between 
the orientations $o$ and $o^-$ of the vector space $\pgot$.

\medskip

If we use Theorems \ref{theo:main-1} and \ref{theo:main-2} we get the following 

\begin{coro}\label{coro:equality-Q-G-K-M}
If the $G$-action on $M$ has {\em abelian infinitesimal stabilizers}, we have 
${\bf r}^o\left(\QSf_G(M)\right)=\QSf_K(M)\otimes \Scal_\pgot^o$.
\end{coro}

\medskip

The following conjecture says that the functorial property of $\QSf$ relative to restrictions (see Theorem \ref{theo:pep-formal-3})
should also holds for non-compact groups.

\begin{conj}
If the $G$-action on $M$ has {\em abelian infinitesimal stabilizers} then the following relation 
$$
\QSf_G(M)\vert_K=\QSf_K(M)
$$
holds in $\wR(K)$.
\end{conj}

\medskip

The remaining part of this section is devoted to the proof of Theorem \ref{theo:main-2}.

\medskip

We work with the manifold $M:=G\times_K N$. We denote $\Phi_K^N:N\to\kgot^*$ the restriction of $\Phi_G:M\to\ggot^*$ to the submanifold $N$. We will use the $K$-equivariant isomorphism $\pgot\times N\simeq M$ defined by 
$(X,n)\mapsto [e^X,n]$.

The maps $\Phi_G,\Phi_K,\Phi_K^N$ are related through the relations $\Phi_G(X,n)= e^X\cdot \Phi_K^N(n)$ and\footnote{$\p_{\kgot,\ggot}:\ggot^*\to\kgot^*$ is the canonical projection.} $\Phi_K(X,n)= \p_{\kgot,\ggot}(e^X\cdot \Phi_K^N(n))$.

We consider the Kirwan vector fields on $N$ and $M$
$$
\kappa_N(n)=-\Phi_K^N(n)\cdot n\quad,\quad \kappa_M(m)=-\Phi_K(m)\cdot m.
$$

The following result is proved in \cite{pep-jems}[Section 2.2].

\begin{lem}\label{lem:Z-N-M}
An element $(X,n)\in \pgot\times N $ belongs to $Z_M:=\{\kappa_M=0\}$ if and only if $X=0$ 
and $n\in Z_N:=\{\kappa_N=0\}$.
\end{lem}

Let us recall how are defined the characters $\QSf_K(M)$ and $\QSf_{\tilde{K}}(N)$. We start with the decomposition 
$Z_N=\coprod_{\beta\in\Bcal} Z_\beta$ where $Z_\beta=K(N^\beta\cap(\Phi_K^N)^{-1}(\beta))$, and $\Bcal=\Phi_K^N(Z_N)\cap \tgot^*_+$.  
Thanks to Lemma \ref{lem:Z-N-M} the corresponding decomposition on $M$ is 
$Z_M:= \coprod_{\beta\in\Bcal} \{0\}\times Z_\beta$.

By definiton we have 
$$
\QSf_K(N):=\sum_{\beta\in\Bcal}\QS_K(N,Z_\beta)\ \in\ \wR(\tK)
$$
and $\QSf_K(M)=\QSf_K(\pgot\times N):=\sum_{\beta\in\Bcal}\QS_K(\pgot\times N,\{0\}\times Z_\beta)\ \in\ \wR(K)$. 
The proof of Theorem \ref{theo:main-2} is completed if we show that for any $\beta\in\Bcal$ we have 
\begin{equation}\label{eq:beta}
\QS_K(\pgot\times N,\{0\}\times Z_\beta)\otimes \overline{\Scal_\pgot^{o^+}}\,=\,\QS_{\tilde{K}}(N, Z_\beta) \quad \in R(\tilde{K}).
\end{equation}

Let $\Scal$ be the $G$-equivariant $\spinc$-bundle on $M$. The $K$-equivariant diffeomorphism $M\simeq \pgot\times N$ induces a 
$\tK$-equivariant isomorphism at the level of $\spinc$ bundles:
$$
\Scal\simeq\Scal_\pgot^{o^+}\otimes \Scal_N.
$$

We denote $\cl_\pgot :\pgot\to\End(\Scal_\pgot)$ the Clifford action associated to the Clifford module $\Scal_\pgot$. 
Any $X\in \pgot$ determines an odd linear map $\cl_\pgot(X): \Scal_\pgot\to \Scal_\pgot$.

For $n\in N$, we denote $\cl_n :\T_n N\to\End(\Scal_N\vert_n)$ the Clifford action associated to the $\spinc$ bundle $\Scal_N$. Any 
$v\in \T_n N$ determines an odd linear map $\cl_n(v): \Scal_N\vert_n\to \Scal_N\vert_n$.

\begin{lem}
Let $U_\beta\subset N$ be a small invariant neighborhood of $Z_\beta$ such that $Z_N\cap\overline{U_\beta}=Z_\beta$. 

$\bullet$ The character $\QS_{\tK}(N, Z_\beta)$ is equal to the index of the $\tK$-transversally elliptic symbol
$$
\sigma^1_{n}(v): \Scal_N^+\vert_n\longrightarrow \Scal_N^-\vert_n, \quad v\in\T_n U_\beta
$$
defined by $\sigma^1_{n}(v)=\cl_n(v+\Phi_K^N(n)\cdot n)$.

$\bullet$ The character $\QS_K(\pgot\times N, \{0\}\times Z_\beta)$ is equal to the index of the $K$-transversally elliptic symbol
$$
\sigma^2_{(A,n)}(X,v): 
(\Scal_\pgot^{o^+}\otimes\Scal_N\vert_n)^+\longrightarrow (\Scal_\pgot^{o^+}\otimes \Scal_N\vert_n)^-
$$
defined by $\sigma^2_{(A,n)}(X,v)=\cl_\pgot(X+[\Phi_K^N(n),A])\otimes \cl_n(v+\Phi_K^N(n)\cdot n)$
for $(X,v)\in\T_{(A,n)}(\pgot\times U_\beta)$. 

\end{lem}

\begin{proof} The first point corresponds to the definition of the character $\QS_{\tK}(N, Z_\beta)$.

By definition, $\QS_K(\pgot\times N, \{0\}\times Z_\beta)$ is equal to the index of the $K$-transversally elliptic symbol
$$
\tau_{(A,n)}(X,v)=\cl_\pgot(X+[\Phi_K(X,n),A])\otimes \cl_n(v+\Phi_K(X,n)\cdot n).
$$ 
It is  not difficult to see that 
$$
\tau^t_{(A,n)}(X,v)=\cl_\pgot(X+[\Phi_K(tX,n),A])\otimes \cl_n(v+\Phi_K(tX,n)\cdot n), \qquad 0\leq t\leq  1,
$$ 
defines an homotopy of transversally elliptic symbols between $\sigma^2=\tau^0$ and $\tau=\tau^1$: like in Lemma \ref{lem:Z-N-M}, we use the fact that 
$[\Phi_K(0,n),A]=0$ only if $A=0$.  It proves the second point. $\Box$
\end{proof}

\medskip

We can now finish the proof of (\ref{eq:beta}). We use here the following isomorphism of 
Clifford modules for the vector space $\pgot\times\pgot$~:
$$
\Scal_\pgot^{o^+}\otimes \overline{\Scal_\pgot^{o^+}}\simeq \bigwedge_\Cbb \pgot_\Cbb,
$$
where the Clifford action $(X,Y)\in \pgot\times\pgot$ on the left is $\cl_\pgot(X)\otimes \cl_\pgot(Y)$ and on the right is 
$\cl_{\pgot_\Cbb}(X+i Y)$.

The product $\sigma^2\otimes \overline{\Scal_\pgot^{o^+}}$ corresponds to the symbol
$$
\cl_\pgot(X+[\Phi_K(X,n),A])\otimes \cl_\pgot(0)\otimes \cl_n(v+\Phi_K^N(n)\cdot n)
$$
which is homotopic to 
$$
\cl_\pgot(X+[\Phi_K(X,n),A])\otimes \cl_\pgot(A)\otimes \cl_n(v+\Phi_K^N(n)\cdot n),
$$
and is also homotopic to 
$$
\sigma^3:=\cl_\pgot(X)\otimes \cl_\pgot(A)\otimes \cl_n(v+\Phi_K^N(n)\cdot n).
$$
We have then proved that the $K$-equivariant index of $\sigma^2$ times $\overline{\Scal_\pgot^{o^+}}\in R(\tK)$ is equal to 
the $\tK$-equivariant index of $\sigma^3$ (that we denote $\indice_{\tK}^{\pgot\times U_\beta}(\sigma^3)$). 
The multiplicative property of the equivariant index \cite{Atiyah74} tells us that 
$$
\indice_{\tK}^{\pgot\times U_\beta}(\sigma^3)=\indice_{\tK}^{\pgot}(\cl_{\pgot_\Cbb}(X+iA))\cdot \indice_{\tK}^{U_\beta}(\sigma^1).
$$
But $\cl_{\pgot_\Cbb}(X+iA):\bigwedge{}^+_\Cbb \pgot_\Cbb \to \bigwedge{}^-_\Cbb \pgot_\Cbb$, $(X,A)\in \T\pgot$,  
is the Bott symbol and its index is equal to the trivial $1$-dimensional representation of $\tK$. We have finally proved that 
the $K$-equivariant index of $\sigma^2$ times $\overline{\Scal_\pgot^{o^+}}$ is equal to the $\tK$-equivariant index of 
$\sigma^1$. The proof of (\ref{eq:beta}) is complete. $\Box$

\subsection{Proof of the main Theorem}

Let $G$ be a connected semi-simple subgroup of $G'$ with finite center, and let $\Ocal'\in \what{G}'_d$. We suppose 
that the representation $\pi_{\Ocal'}^{G'}$ is 
$G$-admissible. Then we have a decomposition 
$$
\pi_{\Ocal'}^{G'}\vert_G=\sum_{\Ocal\in \wG_d} m_\Ocal\, \pi^{G}_\Ocal.
$$

Let $\Phi_G:\Ocal'\to\ggot^*$ be the moment map relative to the $G$-action on $\Ocal'$. We have proved in 
Theorem \ref{theo:equivalence}, that the $G$-admissibility of $\pi_{\Ocal'}^{G'}$ implies the properness of 
$\Phi_G$. Moreover, since $\Ocal'$ is a regular orbit, the $G$-action on it is proper. Finally we see that $\Ocal'$ 
is a $\spinc$ prequantized proper$^2$ Hamiltonian $G$-manifold. We can 
consider its formal $\spinc$ quantization $\QSf_G(\Ocal')\in \wR(G,d)$, which is defined by the relation
$$
\QSf_G(\Ocal'):=\sum_{\Ocal\in \widehat{G}_d} \QS(\Ocal'\slash\!\!\slash \Ocal)\, \pi^G_\Ocal.
$$

Theorem \ref{theo:main} is proved if we show that $\pi_{\Ocal'}^{G'}\vert_G$ and $\QSf_G(\Ocal')$ are equal in $\wR(G,d)$. 
Since the morphism ${\bf r}^o:\wR(G,d)\to \wR(\tilde{K})$ is one to one, it is sufficient to prove that 
\begin{equation}\label{eq:equality-r-o}
{\bf r}^o\left(\pi_{\Ocal'}^{G'}\vert_G\right)={\bf r}^o\left(\QSf_G(\Ocal')\right).
\end{equation}

On one hand, the element ${\bf r}^o\left(\pi_{\Ocal'}^{G'}\vert_G\right)$ is equal to $\pi_{\Ocal'}^{G'}\vert_K\otimes \Scal_\pgot^o$. 
The restriction $\pi_{\Ocal'}^{G'}\vert_K\in \wR(K)$, which is well defined since the moment map $\Phi_K:\Ocal'\to\kgot^*$ is proper, is equal to 
$\QSf_K(\Ocal')$ (see Corollary \ref{coro:restriction-pi-O-G-H}). So we get 
$$
{\bf r}^o\left(\pi_{\Ocal'}^{G'}\vert_G\right)=\QSf_K(\Ocal')\otimes \Scal_\pgot^o.
$$

On the other hand, Corollary \ref{coro:restriction-pi-O-G-H} tells us that 
$$
{\bf r}^o\left(\QSf_G(\Ocal')\right)=\QSf_K(\Ocal')\otimes \Scal_\pgot^o.
$$
Hence we obtain Equality (\ref{eq:equality-r-o}). The proof of Theorem \ref{theo:main} is completed.


{\small

\begin{thebibliography}{99}

\bibitem{Atiyah74} {\sc M.F. Atiyah}, {\em Elliptic operators and
compact groups}, Lecture Notes in Mathematics {\bf 401}, Springer-Verlag, Berlin, 1974.


%
%
%
%


\bibitem{Barbasch-Vogan} {\sc D. Barbasch} and {\sc D.A. Vogan}, {\em The local structure of characters}, J. F. A. {\bf 37} (1980), 37--55.

\bibitem{B-G-V} {\sc N. Berline}, {\sc E. Getzler} and {\sc M. Vergne}, {\em Heat kernels and
Dirac operators}, Grundlehren {\bf 298}, Springer, Berlin, 1991.







%

%
%

%


\bibitem{Duflo80} {\sc M. Duflo}, {\em Construction de repr\'esentations unitaires d'un groupe de Lie}, CIME, Cortona (1980).


\bibitem{Duflo-Vargas2007} {\sc M. Duflo} and {\sc J.A. Vargas},  {\em Proper map and multiplicities}, 2007, preprint.

\bibitem{Duflo-Vargas2010} {\sc M. Duflo} and {\sc J.A. Vargas}, {\em Branching laws for square integrable representations}, Proc. Japan Acad. {\bf 86} (2010),  49--54.

\bibitem{Duflo-Vergne2011} {\sc M. Duflo} and {\sc M. Vergne},  {\em Kirillov's formula and Guillemin-Sternberg conjecture}, C.R.A.S. {\bf 349} (2011), 1213--1217. 





%
%





%
%




\bibitem{Guillemin-Sternberg82} {\sc V. Guillemin} and {\sc S. Sternberg},
{\em Geometric quantization and multiplicities of group representations}, Invent. Math. {\bf 67} (1982), 515--538.





%





\bibitem{Hecht-Schmid} {\sc H. Hecht} and {\sc W. Schmid}, {\em A proof of Blattner’s conjecture}, 
Invent. Math. {\bf 31}, (1975), 129--154. 


\bibitem{Hochs-Song-duke} {\sc P. Hochs}, and {\sc Y. Song}, {\em Equivariant indices of $\spinc$-Dirac operators 
for proper moment maps}, Duke Math. J. {\bf 166}, (2017), 1125--1178.

\bibitem{Hochs-Song-Wu} {\sc P. Hochs}, {\sc Y. Song}, and {\sc S. Yu}, {\em A geometric realisation of tempered representations 
restricted to maximal compact subgroups}, ArXiv:1705.02088.




%


\bibitem{Kawasaki81} {\sc T. Kawasaki}, {\em The index of elliptic operators over V-manifolds}, 
Nagoya Math. J. {\bf 84} (1981), 135--157.


\bibitem{Kirillov62} {\sc A.A. Kirillov}, {\em Unitary representations of nilpotent Lie groups},  Uspekhi Mat. Nauk. 
 {\bf 17} (1962), 53--104.


\bibitem{Kirwan-Invent84} {\sc F. Kirwan}, {\em Convexity properties of the moment mapping III}, 
Invent. Math. {\bf 77} (1984), 547--552.



\bibitem{Toshi-Inventiones94} {\sc T. Kobayashi}, {\em Discrete decomposability of the 
restriction of $A_\mathfrak{q}(\lambda)$ 
with respect to reductive subgroups and its applications}, Invent. Math.  {\bf 117}  (1994),  181--205.

\bibitem{Toshi-Annals98} {\sc T. Kobayashi}, {\em Discrete decomposability of the restriction of $A_\mathfrak{q}(\lambda)$ 
with respect to reductive subgroups. II.  Micro local analysis and asymptotic K-support}, 
Annals of Math. {\bf 147} (1998), 709--729. 

\bibitem{Toshi-Inventiones98} {\sc T. Kobayashi}, {\em Discrete decomposability of the restriction of 
$A_\mathfrak{q}(\lambda)$ with respect to reductive subgroups. III. Restriction of Harish-Chandra 
modules and associated varieties}, Invent. Math. {\bf 131} (1998), 229--256. 



\bibitem{Toshi-Progress-Math} {\sc T. Kobayashi}, {\em Restrictions of unitary representations of real reductive groups}, 
in Lie theory, 139--207, Progr. Math. {\bf 229}, Birkh\"{a}user, Boston, MA, 2005, 139--207.

\bibitem{Kostant70} {\sc B. Kostant}, Quantization and unitary representations, in {\em Modern Analysis and Applications}, Lecture Notes in Math., Vol. 170, Springer-Verlag, 1970, p. 87-207.


\bibitem{L-M-T-W} {\sc E. Lerman}, {\sc E. Meinrenken}, {\sc S. Tolman} and {\sc C. Woodward},
{\em Non-Abelian convexity by symplectic cuts}, Topology {\bf 37} (1998),  245--259.

\bibitem{Ma-Zhang14} {\sc X. Ma} and {\sc W. Zhang}, 
{\em Geometric quantization for proper moment maps: the Vergne conjecture},  Acta Mathematica {\bf 212} (2014),  11--57.

\bibitem{Meinrenken98} {\sc E. Meinrenken}, {\em Symplectic surgery and the Spin\textsuperscript{c}-Dirac operator},
Advances in Math. {\bf 134} (1998),  240--277.

\bibitem{Meinrenken-Sjamaar} {\sc E. Meinrenken} and {\sc R. Sjamaar},
{\em Singular reduction and quantization}, Topology {\bf 38} (1999), 699--762.


\bibitem{pep-RR} {\sc P-E. Paradan}, {\em Localization of the Riemann-Roch character},   J. Functional Analysis  {\bf 187} (2001),  442--509.

\bibitem{pep-ENS} {\sc P-E. Paradan}, {\em Spinc-quantization and the K-multiplicities of the discrete series}, 
Annales scientifiques de l'E.N.S. {\bf 36} (2003), 805–845.


\bibitem{pep-IF} {\sc P-E. Paradan}, {\em Formal geometric quantization},  
Ann. Inst. Fourier  {\bf 59} (2009),  199--238.

\bibitem{pep-pacific} {\sc P-E. Paradan}, {\em Formal geometric quantization II},  
Pacific journal of mathematics, {\bf 253} (2011), 169-211.

\bibitem{pep-JSG} {\sc P-E. Paradan}, {\em Spin quantization commutes with reduction},   J. of Symplectic Geometry  {\bf 10} (2012),  389--422.

\bibitem{pep-jems} {\sc P-E. Paradan}, {\em Quantization commutes with reduction in the non-compact setting: 
the case of holomorphic discrete series}, Journal E.M.S. {\bf } (2015), 

\bibitem{pep-formal3} {\sc P-E. Paradan}, {\em Formal geometric quantization III, Functoriality in the spinc setting},  
Arxiv preprint, 2017.






\bibitem{pep-vergne:witten}{\sc P-E. Paradan}  and {\sc M. Vergne}, {\em Witten non abelian localization  
for equivariant $K$-theory and the $[Q,R]=0$ Theorem}, preprint arXiv 1504.07502 (2015), accepted in Memoirs of the A.M.S..

\bibitem{pep-vergne:magic}{\sc P-E. Paradan}  and {\sc M. Vergne}, {\em Admissible coadjoint orbits for compact Lie groups}, 
preprint 	arXiv:1512.02367 (2015), accepted in Transformations Groups.

\bibitem{pep-vergne-acta}{\sc P-E. Paradan}  and {\sc M. Vergne}, 
{\em Equivariant Dirac operators and differential geometric invariant theory}, 
preprint 	arXiv:1512.02367 (2015), accepted in Acta Mathematica.


\bibitem{Schmid71} {\sc W. Schmid}, On a conjecture of Langlands, 
{\em Ann. of Math.} {\bf 93}, 1971, p. 1-42.

\bibitem{Schmid76}{\sc W. Schmid}, $L\sp{2}$-cohomology and the
  discrete series, {\em Ann. of Math.} {\bf 103}, 1976, p. 375-394.

\bibitem{Souriau} {\sc J.M. Souriau},  {\em Structure des syst\`emes dynamiques}, 
Ma\^\i trise de math\'ematiques, Dunod, 1970.









\bibitem{Tian-Zhang98} {\sc Y. Tian} and {\sc W. Zhang}, {\em An analytic proof of the geometric
quantization conjecture of Guillemin-Sternberg},  Invent. Math. {\bf 132} (1998),  229--259.



%
%
%




\bibitem{Weinstein01} {\sc A. Weinstein}, {\em Poisson geometry of discrete series orbits and momentum 
convexity for noncompact group actions},  Lett. Math. Phys. {\bf 56} (2001),  17--30.






\end{thebibliography}
}
\end{document}